\documentclass{article}
\usepackage{float}
\usepackage{graphicx}
\usepackage{caption}
\graphicspath{{./images/}}
\usepackage{amsfonts}
\usepackage[T2A]{fontenc}
\usepackage[utf8]{inputenc}
\usepackage{hyphenat}
\usepackage{amsmath}
\usepackage{amsthm}
\title{Solvability of equations in elementary functions}
\author{Alexey Kanel-Belov, Alexey Malistov, Rodion Zaytsev (SZU, MIPT, HSE)}
\date{October 2019}
\newtheorem{lemma}{Lemma}[section]
\newtheorem*{remark}{Remark}
\newtheorem{corollary}{Corollary}[section]
\begin{document}

\maketitle

\begin{abstract}
  We prove that the equation $\tan(x) - x = a$ is unsolvable in elementary
functions
\end{abstract}

\section{Introduction}
This article investigates the solvability of the equation $\tan(x) - x = a$ in elementary
functions. In other words, if we start off with functions $e^a$ and $\ln(a)$ along with the identity function, and apply a finite number of arithmetic operations as well as composition, is it possible to get the reverse map of the equation (i.e. $f(a)$ such that $\tan(f(a))-f(a) = a$ for any a)?
Note that all the trigonometric functions as well as powers and their reverse maps can be expressed in terms of the exponent and logarithm, which is why we only need to consider the last two.

The paper was supported by the Russian Science Foundation under grant 17-11-01377.

\section{Topological version of Abel-Ruffini theory}
In this section, some notions from Abel-Ruffini theory will be adopted.

\subsection{General notions}
Suppose we are trying to solve an equation $f(x) = a$ in the complex plane. If $f(x)$ is holomorphic and non-constant, then, by the well-known uniquenuess principle, for any $a$, the set of roots is discrete. In other words, we have several branches of the reverse map. Let now $a$ draw a closed curve in the complex plane. Then the roots also draw curves which, however, aren't necessarily closed. By a well known theorem from complex analysis, if the derivative of a holomorphic function is not zero, then it is a bijection in some neighborhood. Let's call the points where $f'(x) = 0$ for some root critical. Let's only consider curves, without critical points. Then the roots will not will not merge, so that the permutation of roots after $a$ completes the curve is well defined.

\subsection{Induced permutations}
If we change the curve very slightly, and there aren't any critical points very close, the trajectories of the roots will not change much, and, in particular, when $a$ comes back, the permutation of roots will be identical. Therefore, if the area enclosed by the curve doesn't contain any critical points, then we can pull the curve to a constant map, so that none of the roots will have permuted. Therefore, if we want to induce a nontrivial permutation of roots, it suffices to consider curves which move up to a critical point, make a circle around it and then come back. Then, multiplication of these curves (as defined in algebraic topology) will give all achievable permutations of roots.

\subsection{Permutations induced by elementary functions}

\begin{lemma}
	The group induced by elementary functions is cyclic
\end{lemma}

\begin{proof}
Let's calculate the permutations achievable by the logarithm and exponent. Since the exponent is a univalent function, when $a$ moves along a closed path, so does $e^a$. Logarithm, on the other hand, adds $2\pi i$ times the index of the curve.
\end{proof}

\begin{corollary}
	The group induced by elementary functions is abelian
\end{corollary}

\begin{corollary}
	For any $N$, the group induced by elementary functions and their compositions of depth not exceeding $N$ is solvable
\end{corollary}

\begin{proof}
	The commutant of an abelian group is trivial and elementary functions induce an abelian group from the identity curves (i.e. curves where each root draws a closed path). Hence, if we take the identity curves and apply elementary functions to them we will get an abelian group (this follows from the previous corollary). Then, taking its commutant, we get the trivial group again. Applying elementary funcionts to it can only generate an abelian group, so the group generated by functions of depth not greater than two is solvable. Doing this procedure $N$ times (or reasoning by induction) proves the corollary.
\end{proof}

The last corollary means that if we can choose some curves, such that the group generated by their induced permutations is unsolvable, we will have proved the unsolvability in elementary functions.
\begin{remark}

	Calculating the group induced by a given set of curves is rather difficult. This could be automated and done by a computer. Also, if the number of critical points is finite, there is only a finite number of curves, and therefore the whole proof could be done by a machine.
\end{remark}

\begin{lemma}
	The complete group of permutations acts transitively on roots
\end{lemma}

\begin{proof}
If we take any root and move it along any path that connects it with another root, $a(t) = f(x(t))$ will draw a closed curve, because both $x(0)$ and $x(1)$ are roots for the same $a$ (this path, however, must avoid points where $f(x)$ is critical, but since this set is discrete (otherwise $f$ is constant), it is always possible).
\end{proof}

\section{Main result}

In this section, we will investigate the solvability of $\tan(x) - x = a$.

\begin{lemma}
	The critical points of equation $\tan(x) - x = a$ $\pi k, k \in \mathbb{Z}$
\end{lemma}

\begin{proof}
$$(\tan(x) - x)' = \frac{1}{\cos{}^2(x)} - 1 = 0 \Leftrightarrow \cos{}^2(x) = 1 \Leftrightarrow
x \in \pi k, k \in \mathbb{Z}$$
But $\tan(\pi k) - \pi k = -\pi k$ so the critical points for this equation are exactly
$\pi k, k \in \mathbb{Z}$.
\end{proof}

\begin{lemma}
	These critical points are  of third order
\end{lemma}

\begin{proof}
Taking the second derivative, we obtain $\frac{2\sin(x)}{\cos{}^3(x)}$ which also vanishes at critical points. The third derivative, however, doesn't:
$$\frac{2}{\cos{}^2(x)} + \frac{6\sin{}^2(x)}{\cos{}^4} = 1 \text{ if } x = \pi k, k \in \mathbb{Z}$$
\end{proof}

\begin{corollary}
Thus, as it is known from complex analysis, if $a$ makes a loop around any of these critical points, three roots will permute cyclically (and in the same direction as $a$, i.e. if $a$ rotates clockwise, so will the roots).
\end{corollary}

Note that when $a$ comes up to a critical point via real axis, there is only one real root close to the point where the derivative vanishes (because, as can be seen from figure \ref{real_roots} on page \pageref{real_roots}, real roots never merge on the real axis), so that we have one real and non-real roots that permute cyclically. It would therefore be very convenient if there were only two non-real roots in the complex plain altogether. Fortunately, as it turns out, that is the case, the proof of which follows.

\begin{lemma}
$$\tan(z) - z = 0 \text{ has only real roots}$$
\end{lemma}

\begin{proof}
	$$\tan(z) - z = 0 \Leftrightarrow \sin(z) - z\cos(z) = 0 \text{ because }
	\cos(z) = 0 \Rightarrow \sin(z) \neq 0$$ so it can't be a solution, so we can divide the last equation by  $\cos(z)$.
	Let $z = x + iy$, then $$\sin(x + iy) = \sin(x)\cosh(y) + i\cos(x)\sinh(y)$$ so $$|\sin(x + iy)|^2 = \sin^2(x)\cosh^2(y) + \cos^2(x)\sinh^2(y) = \sin^2(x) + \sinh^2(y)$$
	Similarly, $$|\cos(x + iy)|^2 = \cos^2(x)\cosh^2(y) + \sin^2(x)\sinh^2(y) =
	\cos^2(x) + \sinh^2(y)$$

	Recall now Rouché's theorem: if two holomorphic functions $f$ and $g$, holomorphic in some region, are such, that on the boundary of that region $|g(z)| < |f(z)|$, then in that region equations $f(z) = 0$ and $f(z) + g(z) = 0$ have the same number of solutions. Setting $$f(z) = -z\cos(z), g(z) = \sin(z)$$
	and for the region choosing a rectangle with center at zero and right upper corner at
	$$\pi k + iM, k \in \mathbb{N}, M \in \mathbb{R}$$
	the requirements of the theorem are satisfied (if $M$ is large enough).
	Indeed, on the edges on the left and on the right, we have

	\begin{multline}
	|g(z)|^2 = \sin^2(\pi k) + \sinh^2(y) = \sinh^2(y) < \\
	< 1 + \sinh^2(y) = \cos^2(\pi k) + \sinh^2(y) \leq |z\cos(z)|^2 = |f(z)|^2
	\end{multline}

	since
	$$|z| \geq Re \text{ } z = \pi$$
	On the top and on the bottom sides,
	$$|g(z)|^2 \leq 1 + \sinh^2(M)$$
 and
	$$|f(z)|^2 \geq |z|^2\sinh^2(M) \geq M^2\sinh^2(M) > 1 + \sinh^2(M)$$
	for sufficienly large $M$ (we won't need the exact value anyway).
	Hence Rouché's theorem holds and so $tan(z) - z$ has $2k + 1$ roots, because $z\cos(z)$ has roots
	$$\{0,\pm \pi /2, \ldots , \pm (\pi k -  \pi /2) \} $$

	But we know $2k + 1$ roots already: $0$ is a critical point, hence a root of third order, and also we have the intersections of the graph $y = \tan(x)$ with $y = x$ in intervals $(\pi m - \pi /2, \pi m + \pi /2)$, $-k < m < k$, $m \neq 0$, $m \in \mathbb{Z}$ (see figure \ref{real_roots} on page \pageref{real_roots} to visualize this picture).
Therefore, there are no non-real roots, because the rectangle can be chosen as large as needs be.
\end{proof}

We are now ready to prove our main result.
Consider now the equation $\tan(z) - z = a$. If we tend $a$ to zero, we know that we will get three roots merging also in zero.

\begin{center}
	\includegraphics[width=\textwidth]{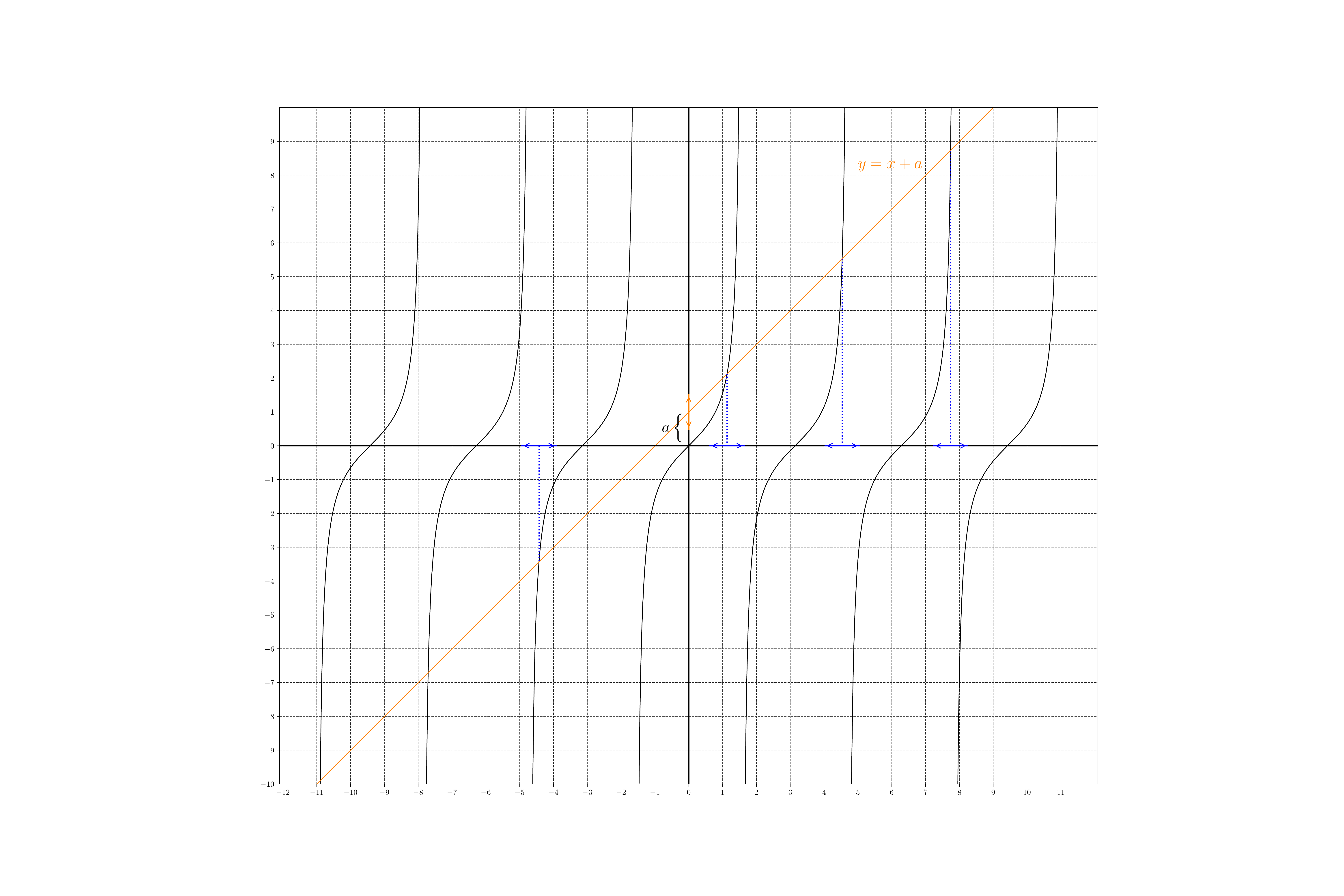}
	\captionof{figure}{Real roots}
	\label{real_roots}
\end{center}

Hence, if $a$ is real and small, but not zero, there will be two non-real roots and one real (in small neighbourhood of zero). Moreover, since the Taylor series for $\tan(x)$ has only real coefficients, the non-real roots will be conjugate. This situation will not change if we move $a$ along the real axis, until we reach a critical point. Then, if we move $a$ around the critical point along the arch of a tiny semicircle, that lies below the real axis, $a$ will rotate about the critical point $180$ degrees anticlockwise. Therefore the roots will permute in such a way that will send the real root either above or below the real axis (strictly speaking this is not yet a permutation, because $a$ has not completed a loop). If the real root has gone below the real axis, then choose the opposite semicircle instead, i.e. the one above the real axis. Then the real root will go above the real axis. Thus, the root above the real axis must have become real, because as $a$ was complex, so must have been the roots. That is, none of them could have crossed the real axis. Therefore, since in the end we have to end up with only one root in each semiplane, the upper being occupied by the root that was real in the beginning, we must have that the one that was in the upper semiplane is now real. The one in the bottom plane has therefore remained there. This procedure could be thought of as simply swapping the real root with the one in the upper semiplane.
Let's now fix a point, say $-1$, move $a$ there, and name the two complex roots $c_1$ and $c_2$, $c_1$ being the one in the upper semiplane, and $c_2$  - its complex conjugate. Name the real roots naturally, starting from zero (i.e. call the real root, that will tend to zero as $a$ does $z_0$, the one to the right $z_1$ and so on). In the beginning, the real roots, starting from the first, are arranged as such: $(z_1, z_2, \ldots)$. Consider now a path starting from $-1$, that goes along the real axis to the left, and as it comes close to a critical point, avoids it along a small semicircle below or above the real axis (choose the one that swaps the real root with the upper one).

\begin{figure}[H]
	\centering
	\includegraphics[width=10cm]{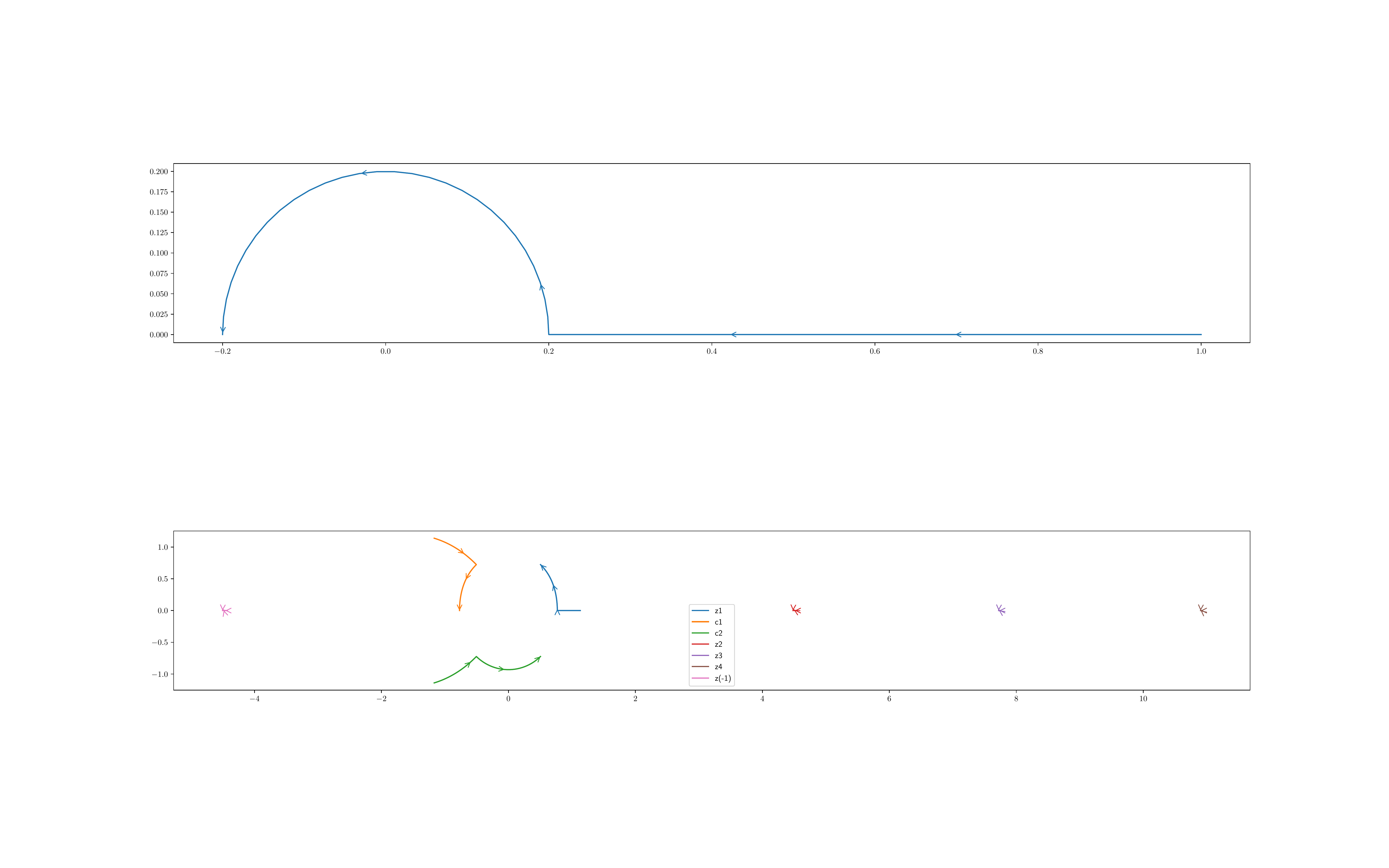}
	\captionof{figure}{The first part of the path}
\end{figure}

\begin{center}
{
\small
In this and the following diagrams, the upper picture is the path drawn by $a$ and the lower one shows the trajectories of the roots.
}
\end{center}

Then it comes up to the $n^{th}$ critical point loops around it in a small circle and comes back along the same path it came up.
Let's see what it does to the roots: when the first critical point is passed, $c_1$ swaps with $z_1$. Then $z_1$ swaps with $z_2$.

\begin{figure}[H]
	\centering
	\includegraphics[width=10cm]{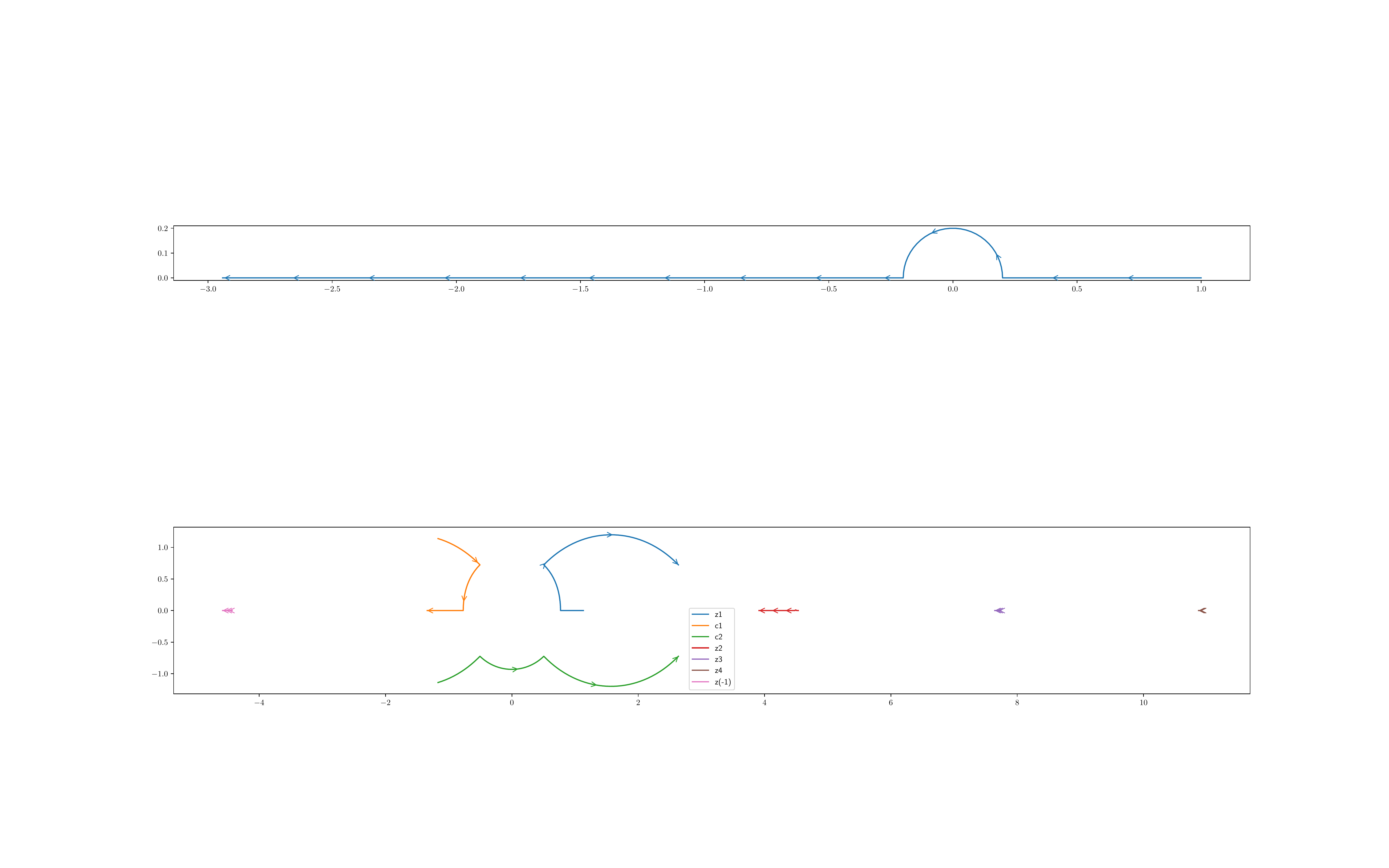}
	\captionof{figure}{The intermediate part of the path}
\end{figure}

Thus, by the time $a$ comes to the $n^{th}$ critical point, the roots will be configured as follows:
$$(c_1, z_1, z_2, \ldots, z_{n-2}, z_n, \ldots)$$

$z_{n-1}$ will be in the upper semiplane, and, as $a$ makes a loop, $z_n$, $z_{n-1}$ and $c_2$ permute cyclically (it can be envisaged as two swaps, because $a$ moves along two semicircles), though we haven't calculated how exactly. But there is no need - we know that the permutation will be non-trivial, and since $A_3$ is generated by any cycle, perhaps making $a$ loop twice, we can create the following configuration:

$$(c_1, z_1, \ldots, z_{n-3}, z_{n-2}, c_2, \ldots)$$
\begin{figure}[H]
	\centering
	\includegraphics[width=10cm]{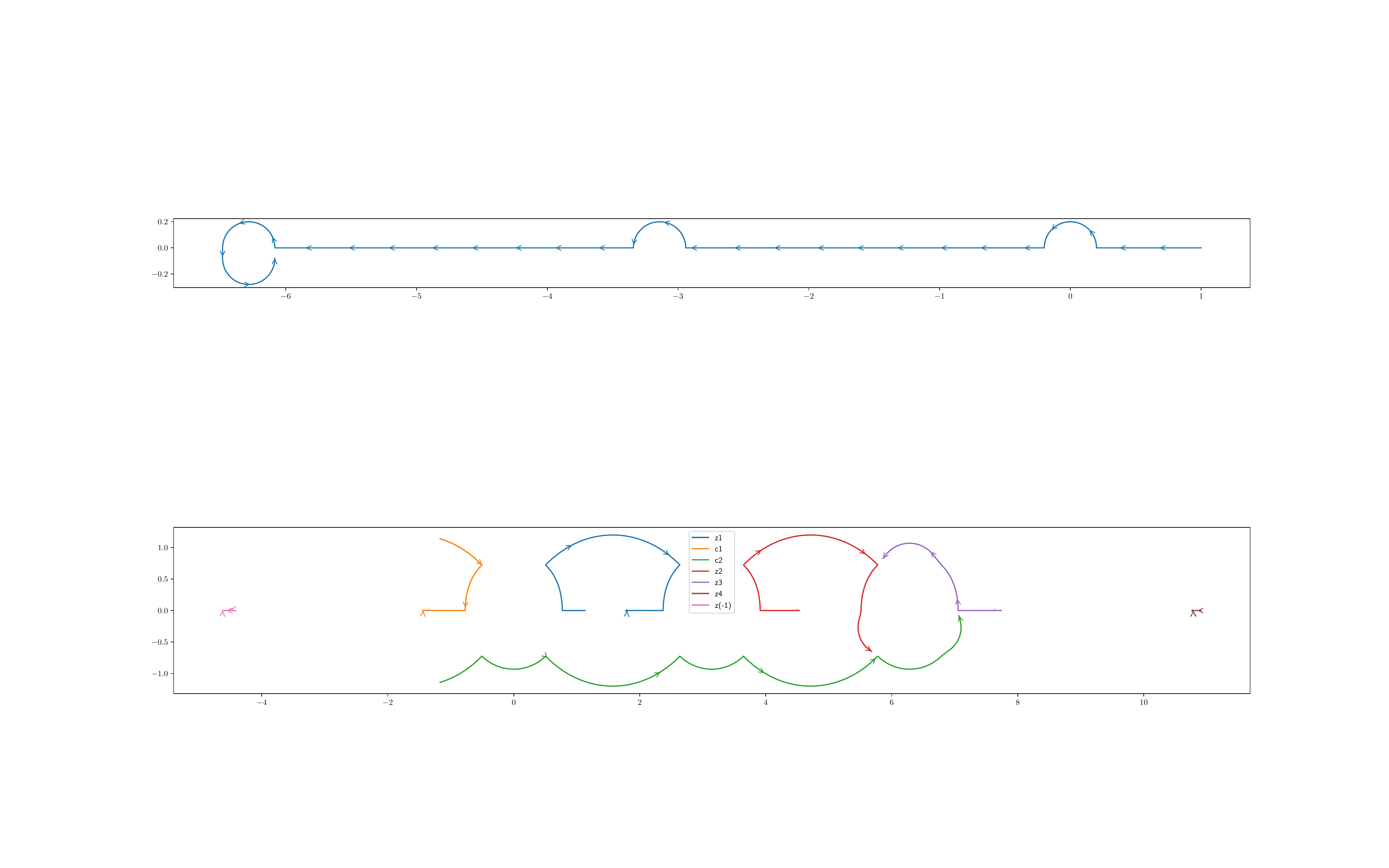}
	\captionof{figure}{The last part of the path}
\end{figure}

Now $z_n$ is hanging in the upper semiplane and $z_{n-1}$ in the lower. Since the path $a$ takes from now is the reverse of the one it took before the loop, all the swaps will be reversed, but since they only swapped two roots at a time, they will simply be the same, swapping the real root with the one in the upper semiplane. Thus, first $z_n$ and $z_{n - 2}$ will be swapped, then $z_{n - 2}$ and $z_{n-3}$ and  so on. Finally, when $a$ completes the path, the configuration will be as such:
$$(z_1, z_2, \ldots, z_{n-2}, z_n, c_2, \ldots)$$
$c_1$ will be back where it was, and $z_{n - 1}$ in the lower semiplane.

\begin{figure}[H]
	\centering
	\includegraphics[width=10cm]{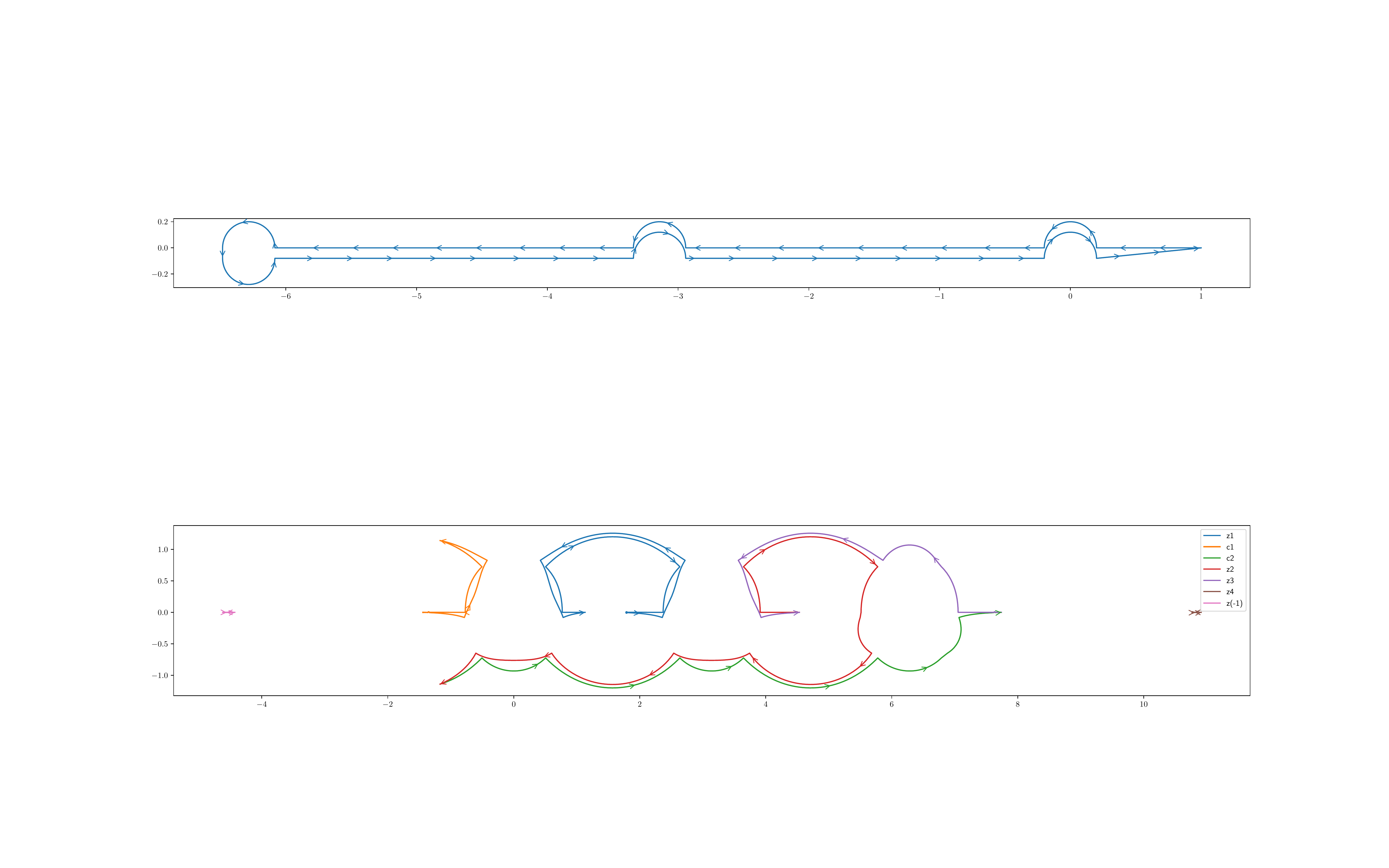}
\end{figure}

If we now do the same loop for the $(n+1)^{th}$ point, the configuration will be
$$(z_1, z_2, \ldots, z_{n - 2}, z_n, z_{n+1}, z_{n-1},\ldots)$$
With both complex roots back where they were. Hence we have a permutation $(n - 1,n+1,n)$. Arather cumbersome calculation (the kind of calculation that is conventionally left as an exercise for the reader with a hint that the alternating group is simple) shows that consecutive 3-cycles generate the alternating group, which is unsolvable. Thus, the equation $\tan(x) - x = a$ is not solvable in elementary functions.

\end{document}